\date{}
\documentclass[oneside,a4paper,12pt,leqno]{article}
\usepackage{amsmath,amssymb,amsopn,amsthm,color,mathrsfs,array,hhline,epsfig}
\usepackage{changebar}

\numberwithin{equation}{section}
\newtheorem{Def}[equation]{Definition}
\newtheorem{LemDef}[equation]{Lemma--Definition}
\newtheorem{Thm}[equation]{Theorem}
\newtheorem{Cor}[equation]{Corollary}
\newtheorem{Prop}[equation]{Proposition}
\newtheorem{Lem}[equation]{Lemma}
\newtheorem{Rem}[equation]{Remark}
\newtheorem{Exer}[equation]{Exercise}
\newtheorem{Example}[equation]{Examples}
\renewenvironment{proof}{\noindent{\it Proof.}\ }{\hfill$\square$\\}

\newenvironment{lemma-definition}{\begin{LemDef}\ \rm}{\end{LemDef}}
\newenvironment{theorem}{\begin{Thm}\ }{\end{Thm}}

\newenvironment{lemma}{\begin{Lem}\ }{\end{Lem}}

\newenvironment{remark}{\begin{Rem}\ }{\end{Rem}}

\newcommand{\maru}[1]{{\ooalign{\hfil#1\/\hfil\crcr\raise.167ex\hbox{\mathhexbox20D}}}} 
\newcommand{\ofrac}[2]{\dfrac{\displaystyle{#1}}{\displaystyle{#2}}}

\newcommand{\underbarl}[1]{\lower 1.4pt \hbox{\underbar{\raise 1.4pt \hbox{#1}}}}

\newcommand{\llg}{\langle\hskip -1.5pt\langle}
\newcommand{\rrg}{\rangle\hskip -1.5pt\rangle}

\newcommand{\jrac}[2]{\dfrac{\lower 2pt\hbox{$#1$}}{\raise 2pt\hbox{$#2$}}}

\newcommand{\twolower}[1]{\hbox{\raise -2pt \hbox{$#1$}}}

\newcommand{\hookdownarrow}{\kern 0pt
\hbox{\vbox{\offinterlineskip \kern 0pt
      \hbox{$\cap$}\kern 0pt
      \hbox{\hskip 3.5pt$\downarrow$}\kern 0pt}}
}  

\newcommand{\lqq}{\lq\lq}

\newcommand{\boxit}[1]{\vbox{\hrule\hbox{\vrule\kern3pt
    \vbox to 43pt{\hsize 182pt\kern3pt#1\eject\kern3pt\vfill}
    \kern1pt\vrule}\hrule}} 

\newcommand{\adots}{\mathinner{\mkern1mu\raise 0pt\vbox{\kern 7pt\hbox{.}}\mkern2mu
    \raise 4pt\hbox{.}\mkern2mu\raise 7pt\hbox{.}\mkern1mu}}
\newcommand{\wdots}{\vbox{\baselineskip 4pt \lineskiplimit 0pt
    \kern 6pt\hbox{.}\hbox{.}\hbox{}}}
\newcommand{\mdots}{\vbox{\baselineskip 4pt \lineskiplimit 0pt
    \kern 6pt\hbox{}\hbox{.}\hbox{.}}}

\renewcommand{\vec}[1]{\ensuremath{\mathchoice
                    {\mbox{\boldmath$\displaystyle#1$}}
                    {\mbox{\boldmath$\textstyle#1$}}
                    {\mbox{\boldmath$\scriptstyle#1$}}
                    {\mbox{\boldmath$\scriptscriptstyle#1$}}}}%
\renewcommand{\i}{\vec{i}}
\newcommand{\lrg}[1]{\langle#1\rangle}

\title{Universal Elliptic Functions\footnote{Date: 2010.3.14}}
\author{Yoshihiro \^Onishi}
\begin{document}
\maketitle
\pagestyle{myheadings}
\noindent
The aim of this note is to discuss the power series expansion at \(u=0\) of the sigma function 
\(\sigma(u)\) of the most general elliptic curve, namely, of
  \begin{equation*}
  \mathscr{E}\ :\ y^2+(\mu_1x+\mu_3)y=x^3+\mu_2x^2+\mu_4x+\mu_6,
  \end{equation*}
where \(\mu_j\) are constants.  
Through out this note, \(\vec{\mu}\) denotes the set consists of the coefficients \(\mu_j\)s.  
We show, in particular, that the series expansion at \(u=0\) of the {\it square} of  \(\sigma(u)\) 
is Hurwitz integral over \(\mathbb{Z}[\vec{\mu}]=\mathbb{Z}[\mu_1,\mu_2,\mu_3,\mu_4,\mu_6]\),
and \(\sigma(u)\) itself is Hurwitz integral over \(\mathbb{Z}[\tfrac{\mu_1}{2},\mu_2,\mu_3,\mu_4,\mu_6]\). 
Namely, any coefficient of the power series expansion of \(\sigma(u)\) at \(u=0\)
is of the form \(c_n\,u^n/n!\) with \(c_n\in\mathbb{Z}[\tfrac{\mu_1}{2},\mu_2,\mu_3,\mu_4,\mu_6]\). 
Although the author had thought that Weierstrass recursion in \cite{Weierstrass} implies directly
Huriwitz integrality of \(\sigma(u)\) at least if \(\mu_1=\mu_2=\mu_3=0\), 
Victor Buchstaber pointed out we do not know how to prove that 
the prime 3 does not appear in the denominator of \(c_n\) above.  
On the other hand, the author had already written this paper because he had been interested in 
why \(c_n\) belongs \(\mathbb{Z}[\tfrac{\mu_1}{2},\mu_2,\mu_3,\mu_4,\mu_6]\)
and does not belong \(\mathbb{Z}[\mu_1,\mu_2,\mu_3,\mu_4,\mu_6]\) in general. 
The method of this paper is completly different from that of Weierstrass, 
which is only an improvement of Nakayashiki's paper \cite{Nakayashiki} 
and is aimed to be generalized for higher genus cases. 

The fact that the power series expansion of the sigma function is not Hurwitz integral only 
at the prime \(2\) would relate with the result of Mazur-Tate \cite{Mazur-Tate}. 

In the last section, we give explicitly \(n\)-plication formula for the coordinate function of 
the curve  \(\mathscr{E}\). 

{\it Acknowledgement}. 
Elena Bunkova read the draft of this paper carefully and corrected several mistakes.  
Buchstaber made the author realize the main result of this paper seems to be new.  
I deeply thank them two. 

\newpage
\section{The Fundamental Differential Form}
\subsection{The most general elliptic curve}
\vskip 5pt
\noindent
Let us consider the most general elliptic curve
  \begin{equation}\label{1.01}
  \mathscr{E}\,:\,y^2+(\mu_1x+\mu_3)y=x^3+\mu_2x^2+\mu_4x+\mu_6.
  \end{equation}
In the sequel, we use notations
  \begin{equation}\label{1.02}
  \begin{aligned}
  f(x,y)&=y^2+(\mu_1x+\mu_3)y-(x^3+\mu_2x^2+\mu_4x+\mu_6), \\
  f_x(x,y)&=\tfrac{\partial}{\partial x}f(x,y)=\mu_1y-(3x^2+2\mu_2x+\mu_4),\\
  f_y(x,y)&=\tfrac{\partial}{\partial y}f(x,y)=2y+(\mu_1x+\mu_3). 
  \end{aligned}
  \end{equation}
We choose local parameter 
\begin{equation}\label{1.25}
t=-x/y 
\end{equation}
at \(\infty\) on \(\mathscr{E}\).  
We never use \(x^{-1/2}\) as a local parameter. 
We call this  \(t\) {\it arithmetic local parameter} of \(\mathscr{E}\). 
We usually express by \(\lrg{t}\) the value determined by each value \(t\).
For instance, the coordinate \(x\) of \(\mathscr{E}\) is denoted by  \(x\lrg{t}\). 
If we introduce also
  \begin{equation}\label{1.26}
  s=1/x,
  \end{equation}
the equation  \(f(x,y)=0\) is rewritten as 
  \begin{equation}\label{1.27}
  s=(1+{\mu_2}s+{\mu_4}s^2+{\mu_6}s^3)t^2+({\mu_1}s+{\mu_3}s^2)t.
  \end{equation}
Using this recursively, we have
  \begin{equation}\label{1.28}
  \begin{aligned}
  s=t^2+{\mu_1}t^3&+({\mu_1}^2+{\mu_2})t^4+({\mu_1}^3+2{\mu_2}{\mu_1}+\mu_3)t^5+\\
   &\ \ \ \ \ \ ({\mu_1}^4+3{\mu_2}{\mu_1}^2+3{\mu_3}{\mu_1}+{\mu_2}^2 +{\mu_4})t^6+\cdots. 
  \end{aligned}
  \end{equation}
By (\ref{1.28}), we see \(x\lrg{t}\) and \(y\lrg{t}\) are expressed as a power series in 
\(\mathbb{Z}[\vec{\mu}][[t]]\) as
  \begin{equation}\label{1.29}
  \begin{aligned}
  x\lrg{t}&=t^{-2}-{\mu_1}t^{-1}-{\mu_2}-{\mu_3}t-({\mu_3}{\mu_1}+{\mu_4})t^2-({\mu_3}{\mu_1}^2+{\mu_4}{\mu_1}+{\mu_2}{\mu_3})t^3+\cdots,\\
  y\lrg{t}&=-t^{-3}+{\mu_1}t^{-2}+{\mu_2}t^{-1}+{\mu_3}+({\mu_3}{\mu_1}+{\mu_4})t+({\mu_3}{\mu_1}^2+{\mu_4}{\mu_1}+{\mu_2}{\mu_3})t^2+\cdots.
  \end{aligned}
  \end{equation}
We choose 
  \begin{equation}\label{1.30}
  \omega_1(x,y)=\frac{dx}{f_y(x,y)}=\frac{dx}{2y+\mu_1x+\mu_3}
  \end{equation}
as the base of the holomorphic 1-form (differentials of the 1st kind) on \(\mathscr{E}\). 
Since
  \begin{equation}\label{1.31}
  \begin{aligned}
  \omega_1(x,y)&=\frac{dx}{2y+\mu_1x+\mu_3}=\frac{\frac{dx}{dt}dt}{2y+\mu_1x+\mu_3}
  \in(1+t\,\mathbb{Z}[\tfrac12,\vec{\mu}][[t]])dt,\\
  \omega_1(x,y)&=-\frac{dy}{f_x(x,y)}\in(1+t\,\mathbb{Z}[\tfrac13,\vec{\mu}][[t]])dt,
  \end{aligned}
  \end{equation}
we have
  \begin{equation}\label{1.32}
  \begin{aligned}
  \omega_1(x,y)
  &=(1+{\mu_1}t+({\mu_2}+{\mu_1}^2)t^2+(2\mu_1\mu_2+2{\mu_3}+{\mu_1}^3)t^3+\cdots)dt\\
  &\ \ \ \ \ \ \in (1+t\,\mathbb{Z}[\vec{\mu}][[t]])dt.
  \end{aligned}
  \end{equation}
\subsection{The fundamental 2-form}
We let consider
  \begin{equation}\label{1.33}
  \Omega(x,y,z,w)=-\frac{y+w+\mu_1z+\mu_3}{z-x}\,\omega_1(x,y)
  =-\frac{(y+w+\mu_1z+\mu_3)dx}
        {(z-x)(2y+\mu_1x+\mu_3)}.
  \end{equation}
This is a 1-form with respect to \((x,y)\) that has simple pole at \((z,w)\) 
and no other poles. 
Note that this is holomorphic at \((z,w+\mu_1z+\mu_3)\). 
Indeed, the numerator becomes \((2w+\mu_1z+\mu_3)=f_y(z,w)\) when \((x,y)=(z,w)\). 
If the value of the local parameter \(t\) gives the value \(x\), 
we denote by \(t'\) the other value of the arithmetic local parameter
which gives the same \(x\) coordinate \(x\lrg{t}\). Hence, 
 \(x\lrg{t}=x\lrg{t'}\). 
Then  \(y\lrg{t}+y\lrg{t'}=-(\mu_1x\lrg{t}+\mu_3)\). So that
  \begin{equation}\label{1.34}
  \begin{aligned}
  t'&=-\ofrac{x\lrg{t'}}{y\lrg{t'}}=\ofrac{x\lrg{t}}{y\lrg{t}+\mu_1x\lrg{t}+\mu_3}\\
    &= -t-{\mu_1}t^2-{\mu_1}^2t^3+(-{\mu_1}^3-{\mu_3})t^4+(-{\mu_1}^4-3{\mu_3}{\mu_1})t^5+\cdots
    \in {t}\,\mathbb{Z}[\mu_1,\mu_3][[t]].
  \end{aligned}
  \end{equation}
The first equality of the above implies
  \begin{equation}\label{1.35}
  \begin{aligned}
  &{tt'}=\ofrac{x\lrg{t}^2}{(y\lrg{t}+\mu_1x\lrg{t}+\mu_3)y\lrg{t}}
   =\ofrac{x\lrg{t}^2}{x\lrg{t}^3+\mu_2x\lrg{t}^2+\mu_4x\lrg{t}+\mu_6}\\
  &=\ofrac{1}{x\lrg{t}(1+\mu_2\tfrac1{x\lrg{t}}+\mu_4\tfrac1{x\lrg{t}^2}+\mu_6\tfrac1{x\lrg{t}^3})}
   =\tfrac{1}{x\lrg{t}}(1-\mu_2\tfrac1{x\lrg{t}}+\cdots). 
  \end{aligned}
  \end{equation}
Finally, we see
\begin{equation}\label{1.36}
\tfrac{1}{x\lrg{t}}=tt'+\mu_2(tt')^2+\cdots\in{tt'}\,\mathbb{Z}[\vec{\mu}][[(tt')]].
\end{equation}
Let us denote as \(x_1=x\lrg{t_1}\), \(y_1=y\lrg{t_1}\). 
Using Weierstrass preparation theorem, 
we define \(p(t_1,t_2)\in(\mathbb{Z}[\vec{\mu}][[t_1,t_2]])^{\times}\) by
\begin{equation}\label{1.364}
\begin{aligned}
{x_2}^{-1}-{x_1}^{-1}=-(t_2-t_1)({t_2}'-t_1)\,p(t_1,t_2). 
\end{aligned}
\end{equation}
Then explicit calculation gives 
\begin{equation}\label{1.365}
\begin{aligned}
p(t_1,t_2)&=1+{\mu_1}{t_1}+{\mu_2}{t_2}^2+({\mu_2}+{\mu_1}^2){t_1}^2+{\mu_1}{\mu_2}{t_2}^3+\cdots\\
&\in {x_1}^{-1}/{t_1}^2+t_2\mathbb{Z}[\vec{\mu}][[t_1,t_2]].
\end{aligned}
\end{equation}
The last one is shown by letting \(t_2=0\). 
On the other hand, we have
\begin{equation}\label{1.37}
\begin{aligned}
y_1+y_2+\mu_1x_2+\mu_3
&=-\frac{x\lrg{t_1}}{t_1}+\frac{x\lrg{{t_2}'}}{{t_2}'}\\
&=-\frac{x\lrg{t_1}}{t_1}+\frac{x\lrg{t_2}}{t_1}-\frac{x\lrg{t_2}}{t_1}+\frac{x\lrg{{t_2}'}}{{t_2}'}\\
&=-\frac1{t_1}(x\lrg{t_1}-x\lrg{t_2})-x\lrg{t_2}\Big(\frac1{t_1}-\frac1{{t_2}'}\Big)\\
\end{aligned}
\end{equation}
and 
\begin{equation}\label{1.375}
\begin{aligned}
&\ \ x_2\Big(\frac{1}{t_1}-\frac{1}{{t_2}'}\Big)\frac1{x_2-x_1}
 =\frac{-{x_1}^{-1}}{{x_2}^{-1}-{x_1}^{-1}}\Big(\frac{1}{t_1}-\frac{1}{{t_2}'}\Big)\\
&=\frac{-{x_1}^{-1}}{(t_2-t_1)({t_2}'-t_1)\,p(t_1,t_2)}
\Big(\frac{1}{t_1}-\frac{1}{{t_2}'}\Big)\\
&=\frac{-{x_1}^{-1}}{(t_2-t_1)({t_2}'-t_1)\,({x_1}^{-1}/{t_1}^2+
\mbox{\lq\lq a series in \(t_2\,\mathbb{Z}[\vec{\mu}][[t_1,t_2]]\)"})}
\frac{{t_2}'-t_1}{t_1{t_2}'}\\
&=\frac{-{x_1}^{-1}}{(t_2-t_1)t_1{t_2}'\,({x_1}^{-1}/{t_1}^2+
\mbox{\lq\lq a series in \(t_2\,\mathbb{Z}[\vec{\mu}][[t_1,t_2]]\)"})}\\
&=\frac{t_1}{(t_2-t_1)t_2}\cdot\frac{t_2}{{t_2}'}\cdot
\frac{-{x_1}^{-1}/{t_1}^2}{({x_1}^{-1}/{t_1}^2+\mbox{\lq\lq a series in \(t_2\,\mathbb{Z}[\vec{\mu}][[t_1,t_2]]\)"})}.\\
\end{aligned}
\end{equation}
Here we note that 
\begin{equation}
\begin{aligned}
{t_2}'/t_2\in -1+t_2\,\mathbb{Z}[\mu_1,\mu_3][[t_2]].
\end{aligned}
\end{equation}
At the last part of (\ref{1.375}), since 
\({x_1}^{-1}/{t_1}^2\in 1+t_1\mathbb{Z}[\vec{\mu}][[t_1]]\), we have
\begin{equation}\label{1.376}         
\begin{aligned}
x_2\Big(\frac{1}{t_1}-\frac{1}{{t_2}'}\Big)\frac1{x_2-x_1}
&=-\Big(\frac{1}{t_2-t_1}-\frac1{t_2}\Big)
(\mbox{\lq\lq a series in \(1+t_2\,\mathbb{Z}[\vec{\mu}][[t_1,t_2]]\)"})\\
&={}+\frac1{t_2}-\frac{1}{t_2-t_1}(\mbox{\lq\lq a series in \(1+t_2\,\mathbb{Z}[\vec{\mu}][[t_1,t_2]]\)"})\\
&\hskip 150pt+(\mbox{\lq\lq a series in \(\mathbb{Z}[\vec{\mu}][[t_1,t_2]]\)"}).
\end{aligned}
\end{equation}
Therefore, 
\begin{equation}\label{1.38}
\begin{aligned}
\frac{y_1+y_2+\mu_1x_2+\mu_3}{x_2-x_1}
&=\frac1{t_1}-\frac{x_2}{x_2-x_1}\Big(\frac1{t_1}-\frac1{{t_2}'}\Big)\\
&=\frac1{t_1}-\frac1{t_2}
+(\mbox{\lq\lq a series in \(\mathbb{Z}[\vec{\mu}][[t_1,t_2]]\)"})\\
&\hskip 50pt +\frac1{t_2-t_1}(\mbox{\lq\lq a series in \(1+t_2\mathbb{Z}[\vec{\mu}][[t_1,t_2]]\)")}.
\end{aligned}
\end{equation}
Defining \(b\lrg{t_1,t_2}\) by
\begin{equation}\label{1.39}
\begin{aligned}
&\Big(\frac{y_1+y_2+\mu_1x_2+\mu_3}{x_2-x_1}-\frac1{t_1}+\frac1{t_2}\Big)\omega_1\lrg{t_1}\\
&=\mathbb{Z}[\vec{\mu}][[t_1,t_2]]\omega_1\lrg{t_1}+\frac1{t_2-t_1}b\lrg{t_1,t_2}dt_1\\
&b\lrg{t_1,t_2}\in\mathbb{Z}[\vec{\mu}][[t_1,t_2]], 
\end{aligned}
\end{equation}
since
\begin{equation}\label{1.40}
\begin{aligned}
\lim_{t_2\to t_1}\frac{x_2-x_1}{y_1+y_2+\mu_1x_2+\mu_3}{\cdot}\frac1{t_2-t_1}
&=\frac{\frac{dx}{dt}\lrg{t_1}}{2y_1+\mu_1x_2+\mu_3}\\
&=\omega_1\lrg{t_1}/dt_1, 
\end{aligned}
\end{equation}
it must be 
\begin{equation}\label{1.41}
b\lrg{t_1,t_1}=1. 
\end{equation}
Thus 
\begin{equation}\label{1.42}
b\lrg{t_1,t_2}\in 1+(t_2-t_1)\mathbb{Z}[\vec{\mu}][[t_1,t_2]].
\end{equation}
We summarize this fact as a theorem:
\begin{theorem}\label{1.43}
We have
\begin{equation}\label{1.44}
\Big(\frac{y_1+y_2+\mu_1x_2+\mu_3}{x_2-x_1}-\frac1{t_1}+\frac1{t_2}\Big)\,\omega_1\lrg{t_1}+\frac{dt_1}{t_1-t_2}
\in\mathbb{Z}[\vec{\mu}][[t_1,t_2]]dt_1.
\end{equation}
\end{theorem}
\noindent
By using a computer, we have first several terms:
\begin{equation}\label{1.45}
\begin{aligned}
\Big(\frac{y+w+\mu_1z+\mu_3}{z-x}&-\frac1{t_1}+\frac1{t_2}\Big)\,\omega_1(x,y)+\frac{dt_1}{t_1-t_2}\\
=\big(&-{\mu_2}{}{t_1}
 -{\mu_3}{}{t_2}{}{t_1}\\
&-({\mu_2}{}{\mu_1}+2{}{\mu_3}){}{t_1}^2\\
&-(2{}{\mu_3}{}{\mu_1}+{\mu_4}){}{t_2}{}{t_1}^2\\
&-({\mu_3}{}{\mu_1}+{\mu_4}){}{t_2}^2{}{t_1}\\
&-({\mu_2}{}{\mu_1}^2+4{}{\mu_3}{}{\mu_1}+{\mu_2}^2+2{}{\mu_4}){}{t_1}^3\\
&-({\mu_3}{}{\mu_1}^2+{\mu_4}{}{\mu_1}+{\mu_2}{}{\mu_3}){}{t_2}^3{}{t_1}\\
&-(2{}{\mu_3}{}{\mu_1}^2+2{}{\mu_4}{}{\mu_1}+{\mu_2}{}{\mu_3}){}{t_2}^2{}{t_1}^2\\
&-(3{}{\mu_3}{}{\mu_1}^2+2{}{\mu_4}{}{\mu_1}+2{}{\mu_2}{}{\mu_3}){}{t_2}{}{t_1}^3\\
&-({\mu_2}{}{\mu_1}^3+6{}{\mu_3}{}{\mu_1}^2+2{}{\mu_2}^2{\mu_1}+4{}{\mu_4}{\mu_1}+6{}{\mu_2}{}{\mu_3}){}{t_1}^4
 -\cdots\big)dt_1.
\end{aligned}
\end{equation}
By using \(\Omega\) in  (\ref{1.33}), we consider a 2-form
  \begin{equation}\label{1.47}
  \vec{\xi}(x,y;z,w)=\tfrac{d}{dx}\Omega(z,w;x,y)dx-\omega_1(z,w)\eta_1(x,y)
  \end{equation}
such that
  \begin{equation}\label{1.48}
  \vec{\xi}(x,y;z,w)=\vec{\xi}(z,w;x,y)
  \end{equation}
with a differential of the 3rd kind \(\eta_1\) that has a pole only at \(\infty\). 
Such an \(\eta_1\) is determined modulo constant multiple of \(\omega_1\). 
A solution is given by
  \begin{equation}\label{1.49}
  \eta_1(x,y)  =\frac{-xdx}{2y+\mu_1x+\mu_3}.
  \end{equation}
The power series expansion with respect to \(t\) is given by
  \begin{equation}\label{1.495}
  \begin{aligned}
  \eta_1(x,y)&=-t^{-2}-{\mu_3}t-(\mu_4+2\mu_1\mu_3)t^2-(2\mu_1\mu_4+2\mu_3\mu_2+3{\mu_1}^2\mu_3)t^3-\cdots\\
  &\in -t^{-2}+t\,\mathbb{Z}[\vec{\mu}][[t]].
  \end{aligned}
  \end{equation}
Under the situation, \(\vec{\xi}\) in (\ref{1.47}) is written as 
  \begin{equation}
  \vec{\xi}=\frac{F(x,y,z,w)dxdz}{(x-z)^2f_y(x,y)f_y(z,w)},
  \end{equation}
where
  \begin{equation}\label{1.50}
  \begin{aligned}
  F(x,y;z,w)&=xz(x+z) 
  +({\mu_1}^2+2{\mu_2})xz
  +{\mu_1}(zy+xw) \\
  &+({\mu_3}{\mu_1}+\mu_4)(x+z) 
  +2yw+{\mu_3}(y+w)+{\mu_3}^2+2\mu_6. 
  \end{aligned}
  \end{equation}
\
\subsection{Legendre relation}
Let us choose two closed paths \(\alpha\) and \(\beta\)  
with their intersection being \(\alpha\cdot\beta=-\beta\cdot\alpha=1\) which generate 
the fundamental group of  \(\mathscr{E}\) and let us define
\begin{equation}\label{2.01}
\omega'=\int_{\alpha}\omega_1(x,y), \ \ 
\omega''=\int_{\beta}\omega_1(x,y), \hskip 30pt
\eta'=\int_{\alpha}\eta_1(x,y), \ \ 
\eta''=\int_{\beta}\eta_1(x,y). \ \ 
\end{equation}
Then we have Legendre relation
\begin{equation}\label{2.02}
\omega''\eta'-\omega'\eta''=2\pi \i,
\end{equation}
where \(\i\) is the imaginary unit.
\vskip 20pt
\section{The sigma function}
\subsection{Construction of the sigma function}
Now, we let
\begin{equation}\label{2.03}
R=\gcd\Big(\mathrm{rslt}_x\big(\mathrm{rslt}_y(f,f_y),\mathrm{rslt}_y(f,f_y)\big),\mathrm{rslt}_y\big(\mathrm{rslt}_x(f,f_y),\mathrm{rslt}_x(f,f_y)\big)\Big).
\end{equation}
Then  \(R\) is a squared element in \(\mathbb{Z}[\vec{\mu}]\). 
So we take a square root of it: \(D=R^{1/2}\). 
More explicitly, if we define
\begin{equation}\label{2.04}
\begin{aligned}
b_2&={\mu_1}^2+4{\mu_2},\ \
b_4=2{\mu_4}+{\mu_1}{\mu_3},\ \
b_6={\mu_3}^2+4{\mu_6},\\
b_8&={\mu_1}^2{\mu_6}+4{\mu_2}{\mu_6}-{\mu_1}{\mu_3}{\mu_4}+{\mu_2}{\mu_3}^2-{\mu_4}^2,
\end{aligned}
\end{equation}
then we let
\begin{equation}\label{2.05}
D=-{b_2}^2{b_8}-8{b_4}^3-27{b_6}^2+9{b_2}{b_4}{b_6}. 
\end{equation}
Now let us define Weierstrass' {\it sigma function} by
\begin{equation}\label{2.06}
\sigma(u)=D^{-1/8}\Big(\frac{\pi}{\omega'}\Big)^{1/2}
\exp\big(-\tfrac12 u^2\eta'{\omega'}^{-1}\big)
\vartheta\bigg[\begin{matrix}\tfrac12 \\\tfrac12\end{matrix}\bigg]({\omega'}^{-1}u\big|{\omega'}^{-1}\omega'').
\end{equation}
This function is characterized up to multiplicative constant 
as a \(\mathrm{SL}_2(\mathbb{Z})\)-invariant Jacobi form under the usual action. 

For each \(u\in\mathbb{C}\), 
there is unique pair of \(u'\) and \(u''\in\mathbb{R}\) 
determined by  \(u=u'\omega'+u''\omega''\). 
We use this notation convention also for each lattice point \(\ell\in\Lambda\), and
we write  \(\ell=\ell'\omega'+\ell''\omega''\). 
Moreover, for \(u\), \(v\in\mathbb{C}\), and  \(\ell\in\Lambda\), we let
\begin{equation}\label{2.07}
L(u,v)=u\,(v'\eta'+v''\eta''), \ \ 
\chi(\ell)=\exp\big(2\pi i(\tfrac12\ell'-\tfrac12\ell''+\tfrac12\ell'\ell'')\big) \ (\in \{1,\,-1\}).
\end{equation}

\begin{lemma}\label{2.065}
The function \(\sigma(u)\) is an entire function and not depends of the choice of \(\alpha\) and \(\beta\).  
Therefore, it will be expanded at the origin in terms of \(\mu_j\)s. 
Moreover, \(\sigma(u)\) has poles of order \(1\) at each point of \(\Lambda\)
and no pole elsewhere, and satisfies
\begin{equation}\label{2.08}
\sigma(u+\ell)=\chi(\ell)\sigma(u)\exp{L(u+\tfrac12\ell,\ell)}.
\end{equation}
\end{lemma}

\begin{proof}
The claim that it does not depend of \(\alpha\) and \(\beta\) is not so easy.
Here, we refer \cite{rademacher}, Chapter 9 and 10. 
The proof is using the transformation property of Dedekind \(\eta\)-function. 
See \cite{Baker}, pp.552--557, and \cite{Igusa}, p.85 and pp.176--183. 
The zeroes of \(\sigma(u)\) is given by calculating the integral 
around the boundary of the regular polygon 
associated to \(\alpha\) and \(\beta\) after taking logarithm of (\ref{2.08}). 
The equation (\ref{2.08}) itself is shown by the translational relation for
the theta series and Legendre relation (\ref{2.01}). 
\end{proof}

In the proof of the first claim above, we know that 
there exists \(8\)th root \(\varepsilon\) of \(1\) such that
\(\sigma(u)=\varepsilon u+O(u^2)\). 
We fix the \(8\)th root in  (\ref{2.06}) as
\begin{equation}\label{2.093}
\sigma(u)=u+O(u^2). 
\end{equation}
For an integral domain \(A\) with characteristic \(0\), and an indeterminate \(t\), we denote by 
\(A\llg{t}\rrg\) the ring of the elements 
  \begin{equation}
  \sum_{j=0}^{\infty}C_j\frac{t^j}{j!}\ \ \mbox{with \(C_j\in A\)}. 
  \end{equation}
We call such a series {\it Hurwitz integral} over \(A\). 
We will finally have the power series expansion of  \(\sigma(u)\)  at the origin as follows:
\begin{theorem}\label{2.095}
The power series expansion of the function \(\sigma(u)^2\) around the origin 
belongs to \(\mathbb{Z}[\mu_1,\mu_2,\mu_3,\mu_4,\mu_6]\llg{u}\rrg\), 
and that of \(\sigma(u)\) belongs to \(\mathbb{Z}[\tfrac{\mu_1}2,\mu_2,\mu_3,\mu_4,\mu_6]\llg{u}\rrg\). 
Its first several terms are given by
\begin{equation}\label{2.10}
\begin{aligned}
\sigma(u)&=u 
+((\tfrac{\mu_1}{2})^2 + \mu_2)\tfrac{u^3}{3!}
+((\tfrac{\mu_1}{2})^4 + 2{\mu_2}(\tfrac{\mu_1}{2})^2 + \mu_3\mu_1 + {\mu_2}^2 + 2\mu_4)\tfrac{u^5}{5!} \\
&\ \ \ \ \ \ + ((\tfrac{\mu_1}{2})^6 + 3{\mu_2}(\tfrac{\mu_1}{2})^4 
+ 6\mu_3(\tfrac{\mu_1}{2})^3 + 3{\mu_2}^2(\tfrac{\mu_1}{2})^2 + 6\mu_4(\tfrac{\mu_1}{2})^2\\
&\hskip 50pt + 6\mu_3\mu_2\tfrac{\mu_1}{2} + {\mu_2}^3 + 6\mu_4\mu_2 + 6{\mu_3}^2 + 24\mu_6)\tfrac{u^7}{7!}
+\cdots.
\end{aligned}
\end{equation}
\end{theorem}
The following is key relation through out this note:
\begin{lemma}\label{2.105}
The sigma function relates with the \(2\)-form \(\vec{\xi}\) by
\begin{equation}\label{2.11}
\ofrac{\sigma\bigg(\int_{\infty}^{(z_1,w_1)}\omega_1-\int_{\infty}^{(x_1,y_1)}\omega_1\bigg)
     \,\sigma\bigg(\int_{\infty}^{(z,w)}\omega_1-\int_{\infty}^{(x,y)}\omega_1\bigg)}
      {\sigma\bigg(\int_{\infty}^{(z,w)}\omega_1-\int_{\infty}^{(x_1,y_1)}\omega_1\bigg)
     \,\sigma\bigg(\int_{\infty}^{(z_1,w_1)}\omega_1-\int_{\infty}^{(x,y)}\omega_1\bigg)}
=\exp\bigg(\int_{(z_1,w_1)}^{(z,w)}\int_{(x_1,y_1)}^{(x,y)}\vec{\xi}\bigg),
\end{equation}
where the integrals in the right hand side are given by jointing 
those in the left hand side. 
\end{lemma}
\begin{proof}
Since we have the same factors in both sides when the point \((x,y)\) or \((z,w)\) varies 
through  \(\alpha\) or \(\beta\), both sides must coincide 
up to multiplicative constant.  
If \((x,y)=(x_1,y_1)\), both sides are \(1\). 
Hence the multiplicative constant is \(1\) as a function of \((z,w)\). 
However, if \((z,w)=(z_1,w_1)\), both sides are \(1\) again. 
Therefore, the multiplicative constant is \(1\). 
\end{proof}
\subsection{Solution to Jacobi's inversion problem}
If we define Weierstrass \(\wp\)-function by
  \begin{equation}\label{2.12}
  \wp(u)=-\tfrac{d^2}{du^2}\log\sigma(u),
  \end{equation}
expanding around the point \((z,w)\)
  \begin{equation}\label{2.13}
  \wp\Big(\int_{\infty}^{(x,y)}\omega_1-\int_{\infty}^{(z,w)}\omega_1\Big)
  =\frac{F(x,y,z,w)}{(x-z)^2},
  \end{equation}
which is obtained by taking 2nd derivative of the logarithm of (\ref{2.11}), 
into power series of a local parameter corresponding the variable \((z,w)\) 
around the point at infinity, we see that, if
\begin{equation}
  u=\int_{\infty}^{(x,y)}\omega_1(x,y),
\end{equation}
then 
  \begin{equation}\label{2.14}
  \wp(u)=x,\ 
  \wp'(u)=2y+\mu_1x+\mu_3. 
  \end{equation}
First several terms of expansion of \(\wp(u)\) is given by
  \begin{equation}
  \begin{aligned}
  \wp(u)= \tfrac{1}{u^2}&-\tfrac{1}{12}{\mu_1}^2-\tfrac{1}{3}{\mu_2}\\
  &+(\tfrac{1}{240}{\mu_1}^4+\tfrac{1}{30}{\mu_2}{\mu_1}^2-\tfrac{1}{10}{\mu_3}{\mu_1}+\tfrac{1}{15}{\mu_2}^2-\tfrac{1}{5}{\mu_4})u^2\\
  &+(-\tfrac{1}{6048}{\mu_1}^6-\tfrac{1}{504}{\mu_2}{\mu_1}^4+\tfrac{1}{168}{\mu_3}{\mu_1}^3+(-\tfrac{1}{126}{\mu_2}^2+\tfrac{1}{84}{\mu_4}){\mu_1}^2\\
  &\hskip 100pt 
    +\tfrac{1}{42}{\mu_2}{\mu_3}{\mu_1}-\tfrac{1}{28}{\mu_3}^2-\tfrac{2}{189}{\mu_2}^3+\tfrac{1}{21}{\mu_4}{\mu_2}-\tfrac{1}{7}{\mu_6})u^4\\
  &+(\tfrac{1}{172800}{\mu_1}^8+\tfrac{1}{10800}{\mu_2}{\mu_1}^6-\tfrac{1}{3600}{\mu_3}{\mu_1}^5+(\tfrac{1}{1800}{\mu_2}^2-\tfrac{1}{1800}{\mu_4}){\mu_1}^4\\
  &\hskip 10pt 
  -\tfrac{1}{450}{\mu_2}{\mu_3}{\mu_1}^3+(\tfrac{1}{300}{\mu_3}^2+\tfrac{1}{675}{\mu_2}^3-\tfrac{1}{225}{\mu_4}{\mu_2}){\mu_1}^2-\tfrac{1}{225}{\mu_2}^2+\tfrac{1}{75}{\mu_4}){\mu_3}{\mu_1}\\
  &\hskip 150pt
   +(\tfrac{1}{675}{\mu_2}^4-\tfrac{2}{225}{\mu_4}{\mu_2}^2+\tfrac{1}{75}{\mu_4}^2)u^6+O(u^8).
  \end{aligned}
  \end{equation}

\subsection{Frobenius-Stickelberger formula}

\begin{lemma}\label{2.148} {\rm (Frobenius-Stickelberger)}\ 
The following equality holds\,{\rm :}
\begin{equation}\label{2.15} 
\frac{\sigma(u+v)\,\sigma(u-v)}{\sigma(u)^2\sigma(v)^2}
=x(u)-x(v).
\end{equation}
\end{lemma}
\begin{proof}
By (\ref{2.08}), the left hand side is a periodic function of 
both of \(u\) and \(v\) with respect  \(\Lambda\). 
On the other hand, since \(\sigma(u)\) has poles of order \(1\) at each point in \(\Lambda\),
the divisors of both hand sides coincide.
After expanding both sides into power series of \(u\), 
they are of the form \(1/u^2+\cdots\). 
Hence, the equality holds. 
\end{proof}
\section{Hurwitz Integrality}
\vskip 8pt
\noindent
Let  \(A\)  be an integral domain with characteristic \(0\). 
We denote by  \(A\llg{u}\rrg\)  is the ring consists of all {\it Hurwitz integral} series
with respect to coefficient ring \(A\). 
\vskip 10pt
Since
  \begin{equation}\label{3.05}
  f(x\lrg{t},y)=(y-y\lrg{t})(y-y\lrg{t'}).
  \end{equation}
for any \(t\), we have \(f_y(x\lrg{t},y)=(y-y\lrg{t})+(y-y\lrg{t'})\), so that
  \begin{equation}\label{3.06}
  f_y(x\lrg{t},y\lrg{t})=y\lrg{t}-y\lrg{t'}.
\end{equation}
This yields the following:
\begin{lemma}\label{3.07}
We have
  \begin{equation}\label{3.08}
  \begin{aligned}
  f_y(x\lrg{t},y\lrg{t})&=\frac1{(tt')^3}(t-t')(t^2+\mbox{\lq\lq higher terms in \(\mathbb{Z}[\vec{\mu}][[t]]\)"}).
  \end{aligned}
\end{equation}
\end{lemma}
\noindent
In this section, we prove Theorem \ref{2.095}. 
When
\begin{equation}\label{3.10}
u=\int_{\infty}^{(x,y)}\omega_1,
\end{equation}
by (\ref{1.32}), we see 
\begin{equation}\label{3.11}
u=t+O(t^2)\ \ \ \mbox{in}\ \mathbb{Z}[\vec{\mu}]\llg{t\rrg}, \ \ 
t=u+O(u^2)\ \ \ \mbox{in}\ \mathbb{Z}[\vec{\mu}]\llg{u\rrg}. 
\end{equation}
From (\ref{1.44}), (\ref{1.47}), (\ref{1.495}), and (\ref{1.32}), we have immediately, that 
  \begin{equation}
  \vec{\xi}\lrg{t_1,t_2}-\frac{dt_1dt_2}{({t_1}-{t_2})^2}\in\mathbb{Z}[\vec{\mu}][t_1,t_2]dt_1dt_2.
  \end{equation}
Explicitly, first several terms of this expansion are given by
\begin{equation}\label{3.12}
\begin{aligned}
\vec{\xi}\lrg{t_1,t_2}
=\big((&{t_1}-{t_2})^{-2}+{\mu_3}({t_1}+{t_2})
 +(3{\mu_3}{\mu_1}+2{\mu_4}){t_1}{t_2}\\
&+(2{\mu_3}{\mu_1}+ {\mu_4})({t_1}^2+{t_2}^2)\\
&+(5{\mu_3}{\mu_1}^2+4{\mu_4}{\mu_1}+3{\mu_2}{\mu_3})({t_1}^2{t_2}+{t_1}{t_2}^2)\\
&+(3{\mu_3}{\mu_1}^2+2{\mu_4}{\mu_1}+2{\mu_2}{\mu_3})({t_1}^3+{t_2}^3)\\
&+(8{\mu_3}{\mu_1}^3+7{\mu_4}{\mu_1}^2+11{\mu_2}{\mu_3}{\mu_1}+3{\mu_3}^2+4{\mu_4}{\mu_2}+3{\mu_6}){t_1}^2{t_2}^2\\
&+(7{\mu_3}{\mu_1}^3+6{\mu_4}{\mu_1}^2+10{\mu_2}{\mu_3}{\mu_1}+4{\mu_3}^2+4{\mu_4}{\mu_2}+4{\mu_6})({t_1}^3{t_2}+{t_1}{t_2}^3)\\
&+(4{\mu_3}{\mu_1}^3+3{\mu_4}{\mu_1}^2+ 6{\mu_2}{\mu_3}{\mu_1}+3{\mu_3}^2+2{\mu_4}{\mu_2}+2{\mu_6})({t_1}^4+{t_2}^4)\\
&+\cdots\big)dt_1dt_2,
\end{aligned}
\end{equation}
and
  \begin{equation}\label{3.13}
  \begin{aligned}
  &\int_{{t_1}'}^{{t_2}'}\int_{t_1}^{t_2}\vec{\xi}\lrg{T_1,T_2}dT_1dT_2\\
  &=-\log\Big({-}\ofrac{({t_2}'-t_1)(t_2-{t_1}')}{({t_2}'-t_2)(t_1-{t_1}')}\Big)
  +\frac{\mu_3}2\big(({{t_2}'}^2-{{t_1}'}^2)(t_2-t_1)+({t_2}^2-{t_1}^2)({t_2}'-{t_1}')\big)\\
  &\hskip 50pt +\mbox{\lq\lq a series divisible by \((t_1-t_2)^2(t_1,t_2)^2\) in \(\mathbb{Z}[\vec{\mu}]\llg{t_1,t_2}\rrg\)"}
  \end{aligned}
  \end{equation}
By plugging 
  \begin{equation}\label{3.138}
  \begin{aligned}
  (x,y)&=(x\lrg{{t_2}'},y\lrg{{t_2}'})=(x(-v),y(-v))=(x(v),-y(v)-\mu_1x(v)-\mu_3),\\
  (z,w)&=(x\lrg{t_2},y\lrg{t_2})=(x(v),y(v)),\\
  (x_1,y_1)&=(x\lrg{{t_1}'},y\lrg{{t_1}'})=(x(-u),y(-u))=(x(u),-y(u)-\mu_1x(u)-\mu_3),\\
  (z_1,w_1)&=(x\lrg{t_1},y\lrg{t_1})=(x(u),y(u)),
  \end{aligned}
  \end{equation}
in (\ref{2.11}), we have
  \begin{equation}\label{3.139}
  \frac{\sigma(2u)\sigma(2v)}{\sigma(u+v)^2}
  =\exp\Big(\int_{t_1}^{t_2}\!\!\int_{{t_1}'}^{{t_2}'}\vec{\xi}\lrg{T_1,T_2}dT_1dT_2\Big)
  \end{equation}
Let \(u\) and \(v\) are analytic coordinates corresponding \(t_1\) and \(t_2\), respectively,
we have
\begin{equation}\label{3.14}
\begin{aligned}
\big(x(u)-x(v)\big)^2
&=\bigg(\frac{\sigma(u+v)\sigma(u-v)}{\sigma(u)^2\sigma(v)^2}\bigg)^2
  \ \ \ \mbox{(\(\because\) (\ref{2.15}))}\\ 
&=\frac{\sigma(u+v)^2}{\sigma(2u)\sigma(2v)}
  \frac{\sigma(2u)}{\sigma(u)^4}\frac{\sigma(2v)}{\sigma(v)^4}\sigma(u-v)^2\\
&=\exp\Big({-}\int_{{t_1}'}^{{t_2}'}\int_{t_1}^{t_2}\vec{\xi}\lrg{T_1,T_2}dT_1dT_2\Big)
f_2\lrg{t_1}f_2\lrg{t_2}\sigma(u-v)^2\\
&\hskip 200pt  \ \ \ \mbox{(\(\because\) (\ref{2.11}))}\\
&=\exp\Big({-}\log\ofrac{({t_2}'-t_2)(t_1-{t_1}')}{({t_2}'-t_1)(t_2-{t_1}')}\\
& \hskip 50pt+\mbox{\lq\lq a series in \(\mathbb{Z}[\vec{\mu}]\llg{t_1,t_2}\rrg\)"}\Big)f_y\lrg{t_1}f_y\lrg{t_2}\sigma(u-v)^2\\
&\hskip -50pt
=\bigg(\ofrac{({t_2}'-t_1)(t_2-{t_1}')}{({t_2}'-t_2)(t_1-{t_1}')}
 \times\mbox{\lq\lq a series of the form \(1+\cdots\) in \(\mathbb{Z}[\vec{\mu}]\llg{t_1,t_2}\rrg\)"}\bigg)\\
&\hskip 100pt \times f_y\lrg{t_1}f_y\lrg{t_2}\sigma(u-v)^2.
\end{aligned}
\end{equation}
We here recall (\ref{1.364}):
  \begin{equation}\label{3.20}
  \begin{aligned}
  x\lrg{t_2}-x\lrg{t_1}
  &=\frac{({t_2}'-t_1)(t_2-t_1)\,p(t_1,t_2)}
  {x\lrg{t_1}^{-1}x\lrg{t_2}^{-1}},\\
  p(t_1,t_2)&=1+{\mu_1}{t_1}+{\mu_2}{t_2}^2+({\mu_2}+{\mu_1}^2){t_1}^2+{\mu_1}{\mu_2}{t_2}^3+\cdots.
  \end{aligned}
  \end{equation}
Exchanging  \(t_1\) and \(t_2\), we see
  \begin{equation}\label{3.21}
  x\lrg{t_2}-x\lrg{t_1}
  =-\frac{x\lrg{t_2}^{-1}-x\lrg{t_1}^{-1}}{x\lrg{t_1}^{-1}x\lrg{t_2}^{-1}}
  =-\frac{({t_1}'-t_2)(t_1-t_2)\,p(t_2,t_1)}
  {x\lrg{t_1}^{-1}x\lrg{t_2}^{-1}}.
  \end{equation}
Dividing both sides of  (\ref{2.15}) by \(u-v\) and 
(\ref{2.093}) imply
  \begin{equation}\label{3.215}
  \frac{\sigma(2u)}{\sigma(u)^4}=\frac{d}{du}x(u)=1\Big/\frac{du}{dx}=f_2(x(u),y(u)).
  \end{equation}
Therefore, 
  \begin{equation}\label{3.22}
  \begin{aligned}
  \frac{\sigma(2u)}{\sigma(u)^4}
  &=f_2\lrg{t}=f_2(x\lrg{t},y\lrg{t})
   =y\lrg{t}-y\lrg{t'}
   =-\frac{y\lrg{t'}^{-1}-y\lrg{t}^{-1}}{y\lrg{t}^{-1}y\lrg{t'}^{-1}}\\
  &=-\frac{x\lrg{t}}{t}+\frac{x\lrg{t}}{t'}
   =(t-t')\,\frac{x\lrg{t}}{tt'}.
\end{aligned}
\end{equation}

\newpage
Now we have arrived at the main result as follows:
\begin{theorem}\label{3.245}
Let \(u\) and  \(v\) are analytic coordinates corresponding two values \(t_1\) and \(t_2\) of 
the arithmetic parameter in  {\rm (\ref{1.25})}. 
Then the function \(\sigma(u-v)^2\) is written as a product of formal power series\,{\rm :}
  \begin{equation}\label{3.25}
  \sigma(u-v)^2
  =(t_2-t_1)^2\,q(t_1)q(t_2)\,p(t_1,t_2)\,p(t_2,t_1)\,r(t_1,t_2),
  \end{equation}
where
  \begin{equation}\label{3.26}
  \begin{aligned}
  p&(t_1,t_2)=\frac{x\lrg{t_2}^{-1}-x\lrg{t_1}^{-1}}{({t_2}'-t_1)(t_2-t_1)}
             =1+{\mu_1}{t_1}+{\mu_2}{t_2}^2+({\mu_2}+{\mu_1}^2){t_1}^2+\cdots\\
   &\hskip 50pt \in 1+(t_1,t_2)\,\mathbb{Z}[\vec{\mu}][[t_1,t_2]],\\
  q&(t)=-x\lrg{t}{t}{t'}
       =1-\mu_2t^2-\mu_2\mu_1t^3-(\mu_2{\mu_1}^2+\mu_4)t^4\\
   &\ \ \ \ \ -(\mu_2{\mu_1}^3+2\mu_4\mu_1+\mu_2\mu_3)t^5+\cdots\ 
    \hskip 20pt \in 1+t\,\mathbb{Z}[\vec{\mu}][[t]],  \\
  r&(t_1,t_2)=\exp\bigg[\int_{{t_1}'}^{{t_2}'}\int_{t_1}^{t_2}
       \Big(\vec{\xi}\lrg{t_1,t_2}-\frac{1}{(t_2-t_1)^2}\Big)\bigg]\\
  &=1-(\tfrac{1}{12}{\mu_1}{\mu_3}+\tfrac{1}{6}{\mu_4})(t_1-t_2)^4
     -(\tfrac{1}{6}{\mu_1}^2{\mu_3}+\tfrac{1}{3}{\mu_4}{\mu_1})(t_1-t_2)^4(t_1+t_2)\\
  &\ \  
     +\big(-(\tfrac{1}{30}{\mu_3}^2+(\tfrac{43}{180}{\mu_1}^3+\tfrac{11}{90}{\mu_2}{\mu_1}){\mu_3}+\tfrac{43}{90}{\mu_4}{\mu_1}^2+\tfrac{11}{45}{\mu_2}{\mu_4}+\tfrac{2}{15}{\mu_6})({t_1}^4+{t_2}^4)\\
  &\ \ +(\tfrac{2}{15}{\mu_3}^2+(\tfrac{11}{90}{\mu_1}^3+\tfrac{7}{45}{\mu_2}{\mu_1}){\mu_3}+\tfrac{11}{45}{\mu_4}{\mu_1}^2+\tfrac{14}{45}{\mu_2}{\mu_4}+\tfrac{8}{15}{\mu_6}){t_1}{t_2}({t_1}^2+{t_2}^2)\\
  &\ \ +(-\tfrac{1}{5}{\mu_3}^2+(\tfrac{7}{30}{\mu_1}^3-\tfrac{1}{15}{\mu_2}{\mu_1}){\mu_3}+\tfrac{7}{15}{\mu_4}{\mu_1}^2-\tfrac{2}{15}{\mu_2}{\mu_4}+\tfrac{1}{5}{\mu_6}){t_1}^2{t_2}^2\big)({t_1}-{t_2})^2\\
  &\ \ +\cdots
  \ \ \  \in 1+(t_1-t_2)^2(t_1,t_2)^2\,\mathbb{Z}[\vec{\mu}]\llg{t_1,t_2}\rrg.
  \end{aligned}
  \end{equation}
Since \(t_1=u+\cdots\in\mathbb{Z}[\vec{\mu}]\llg{u}\rrg\) and 
\(t_2=v+\cdots\in\mathbb{Z}[\vec{\mu}]\llg{v}\rrg\), 
we have  
  \begin{equation}\label{3.27}
  \sigma(u-v)^2\in(u-v)^2(\mbox{\lq\lq a series of the form \(1+\cdots\)\, in \,\(\mathbb{Z}[\vec{\mu}]\llg{u,v}\rrg\)}").
  \end{equation}
\end{theorem}

What happens if we take a square root of \(\sigma(u-v)^2\) is known by the following Lemma. 
\begin{lemma}\label{3.01}
Let \(A\) be an integral domain with characteristic \(0\). 
Let \(a_j\in A\) for \(j\geqq1\) and 
  \begin{equation}\label{3.02}
  h(z)=1+2a_1\frac{z}{1!}+2a_2\frac{z^2}{2!}+2a_3\frac{z^3}{3!}+\cdots,
  \end{equation}
be a power series with respect to an indeterminate \(z\). 
Then a power series \(\varphi(z)\) such that
  \begin{equation}\label{3.03}
  h(z)=\varphi(z)^2
  \end{equation}
belongs to \(A\llg z\rrg\). 
\end{lemma}
\begin{proof}
Since
  \begin{equation}\label{3.04}
  \begin{aligned}
  &\Big(1+2a_1\frac{z}{1!}+2a_2\frac{z^2}{2!}+\cdots\Big)^{-\frac12}\\
  &=1-\frac1{1!}\frac12\Big(2{a_1}z+{2a_2}\frac{z^2}{2!}+{2a_3}\frac{z^3}{3!}+\cdots\Big)
  +\frac1{2!}\frac12\frac32\Big(2{a_1}z+{2a_2}\frac{z^2}{2!}+{2a_3}\frac{z^3}{3!}+\cdots\Big)^2\\
  &\ \ \ \ \ -\frac1{3!}\frac12\frac32\frac52
  \Big(2{a_1}z+{2a_2}\frac{z^2}{2!}+{2a_3}\frac{z^3}{3!}+\cdots\Big)^3+\cdots\\
  &=1-\frac1{1!}\Big({a_1}z+{a_2}\frac{z^2}{2!}+{a_3}\frac{z^3}{3!}+\cdots\Big)
  +\frac1{2!}{\cdot}1{\cdot}3\Big({a_1}z+{a_2}\frac{z^2}{2!}+{a_3}\frac{z^3}{3!}+\cdots\Big)^2\\
  &\ \ \ -\frac1{3!}{\cdot}1{\cdot}3{\cdot}5\Big({a_1}z+{a_2}\frac{z^2}{2!}+{a_3}\frac{z^3}{3!}+\cdots\Big)^3+\cdots,
  \end{aligned}
  \end{equation}
the statement follows. 
\end{proof}

It is obvious 
that the expansion of \(q(t)\) and \(p(0,t)\) are 
of the form (\ref{3.02}) for \(A=\mathbb{Z}[\mu_1,\mu_2,\mu_3,\mu_4,\mu_6]\llg{t}\rrg\). 
However, the expansion of \(p(t,0)\) is
of the form (\ref{3.02}) only for \(A=\mathbb{Z}[\frac{\mu_1}2,\mu_2,\mu_3,\mu_4,\mu_6]\llg{t}\rrg\).  
Since the denominators of the coefficients of \(r(t)\) are 
come from the double integral in (\ref{3.26}) with taking exponential, 
and the series of the double integral starts at a term of degree \(4\), 
we see that the expansion of \(r(t)\) is also of the form (\ref{3.02}) 
for \(A=\mathbb{Z}[\mu_1,\mu_2,\mu_3,\mu_4,\mu_6]\llg{t}\rrg\). 
Therefore, \ref{3.01} implies that
  \begin{equation}\label{3.273}
  \sigma(u)\in\mathbb{Z}[\tfrac{\mu_1}2,\mu_2,\mu_3,\mu_4,\mu_6]\llg{t}\rrg
             =\mathbb{Z}[\tfrac{\mu_1}2,\mu_2,\mu_3,\mu_4,\mu_6]\llg{u}\rrg.
  \end{equation}
If we compute this explicitly, we have
  \begin{equation}\label{3.28}
  \begin{aligned}
  \sigma(u)=
   t&+\tfrac{1}{2}{\mu_1}t^2
  +(\tfrac{3}{2}(\tfrac{\mu_1}{2})^2+\tfrac{1}{2}{\mu_2})t^3\\
 &+(\tfrac{5}{2}(\tfrac{\mu_1}{2})^3+\tfrac{3}{2}{\mu_2}(\tfrac{\mu_1}2)+\tfrac{1}{2}{\mu_3})t^4\\
 &+(\tfrac{35}{8}(\tfrac{\mu_1}{2})^4+\tfrac{15}{4}{\mu_2}(\tfrac{\mu_1}{2})^2+\tfrac{29}{16}{\mu_3}(\tfrac{\mu_1}{2})+(\tfrac{3}{8}{\mu_2}^2+\tfrac{5}{12}{\mu_4}))t^5\\
 &+(\tfrac{63}{8}(\tfrac{\mu_1}{2})^5+\tfrac{35}{4}{\mu_2}(\tfrac{\mu_1}{2})^3+\tfrac{25}{3}{\mu_3}(\tfrac{\mu_1}{2})^2+(\tfrac{15}{8}{\mu_2}^2+\tfrac{25}{12}{\mu_4})\tfrac{\mu_1}{2}+\tfrac{5}{4}{\mu_2}{\mu_3})t^6\\
 &+(\tfrac{231}{16}(\tfrac{\mu_1}{2})^6+\tfrac{315}{16}{\mu_2}(\tfrac{\mu_1}2)^4+\tfrac{8941}{360}{\mu_3}(\tfrac{\mu_1}{2})^3+
  +(\tfrac{105}{16}{\mu_2}^2+\tfrac{2641}{360}{\mu_4})(\tfrac{\mu_1}2)^2\\
 &\hskip 50pt +\tfrac{3091}{360}{\mu_2}{\mu_3}(\tfrac{\mu_1}2)
  +\tfrac{103}{120}{\mu_3}^2+\tfrac{5}{16}{\mu_2}^3+\tfrac{391}{360}{\mu_4}{\mu_2}+\tfrac{13}{30}{\mu_6})t^7
  +O(t^8)\\
 =u+&\big((\tfrac{\mu_1}{2})^2+\mu_2\big)\tfrac{u^3}{3!}
  +\mbox{\lqq higher terms in \ \(\mathbb{Z}[\tfrac{\mu_1}2,\mu_2,\mu_3,\mu_4,\mu_6]\llg{u}\rrg\)"}.
  \end{aligned}
  \end{equation}
This result is no other than (\ref{2.10}). 

\begin{remark}\label{3.275}
{\rm 
The author think that this result would be relate
the result in  \cite{Mazur-Tate}. 
}
\end{remark}
\newpage
\section{\(n\)-plication formula}
\vskip 5pt
\noindent
We mention here a small application. 
Using the expansion of the sigma function of the previous section, 
we can compute \(n\)-plication formula as follows.
\vskip 5pt
\noindent
\underbar{If \(n\) is odd}
\begin{equation}\label{4.01}
\begin{aligned}
\psi_n(u):&=\frac{\sigma(nu)}{\sigma(u)^{n^2}}
=n{}x(u)^{\frac{n^2-1}2}
+C_1\,{}x(u)^{\frac{n^2-5}2}{}y(u)
+C_2\,{}x(u)^{\frac{n^2-3}2}\\ 
&\hskip 80pt 
+C_3\,{}x(u)^{\frac{n^2-7}2}{}y(u)
+C_4\,{}x(u)^{\frac{n^2-5}2}+\cdots+C_{n^2-1}
\end{aligned}
\end{equation}
is the \(n\)-plication polynomial, 
namely, the roots  \((x(u),y(u))\)s  of this polynomial  
are just the \(n\)-torsion points of \(\mathscr{E}\). 
Comparing power series expansions of both sides with respect to \(u\), 
we have explicit form of the coefficients \(C_j\)s.
First several terms of them are given by
\begin{equation}\label{4.02}
\begin{aligned}
C_1&=0\\
C_2&={\tfrac1{24}}\,n(n^2-1){\mu_1}^2
+{\tfrac1{6}}\,n(n^2-1){\mu_2},\\
C_3&=0,\\
C_4&={\tfrac1{1920}}\,n(n^2-1)(n^2-9){\mu_1}^4
 +{\tfrac1{240}}\,n(n^2-1)(n^2-9){\mu_2}{\mu_1}^2\\
&+{\tfrac1{120}}\,n(n^2-1)(n^2+6){\mu_3}{\mu_1}
 +{\tfrac1{120}}\,n(n^2-1)(n^2-9){\mu_2}^2
 +{\tfrac1{60}}\,n(n^2-1)(n^2+6){\mu_4},\\
C_5&=0,\\
C_6&={\tfrac1{322560}}\,n(n^2-1)(n^2-3^2)(n^2-5^2){\mu_1}^6
 +{\tfrac1{26880}}\,n(n^2-1)(n^2-3^2)(n^2-5^2){\mu_2}{\mu_1}^4\\
&+{\tfrac1{6720}}\,n(n^2-1)(n^2-3^2)(n^2+10){\mu_3}{\mu_1}^3
 +{\tfrac1{6720}}\,n(n^2-1)(n^2-3^2)(n^2-5^2){\mu_2}^2{\mu_1}^2\\
&+{\tfrac1{3360}}\,n(n^2-1)(n^2-3^2)(n^2+10){\mu_4}{\mu_1}^2
 +{\tfrac1{1680}}n(n^2-1)(n^2-3^2)(n^2+10){\mu_3}{\mu_2}{\mu_1}\\
&+{\tfrac1{5040}}\,n(n^2-1)(n^2-3^2)(n^2-5^2){\mu_2}^3
 +{\tfrac1{840}}\,n(n^2-1)(n^2-3^2)(n^2+10){\mu_4}{\mu_2}\\
&+{\tfrac1{840}}\,n(n^2-1)(n^4+n^2+15){\mu_3}^2
 +{\tfrac1{210}}\,n(n^2-1)(n^4+n^2+15){\mu_6}.
\end{aligned}
\end{equation}
Note that these are elements of \(\mathbb{Z}[\vec{\mu}]\) if  \(n\) is odd.
\vskip 5pt
\noindent
\underbar{If  \(n\) is even}
\begin{equation}\label{4.05}
\begin{aligned}
\psi_n(u):&=\frac{\sigma(nu)}{\sigma(u)^{n^2}}
 =n{}x(u)^{\frac{n^2-4}2}{}y(u)
+C_1\,{}x(u)^{\frac{n^2-2}2}
+C_2\,{}x(u)^{\frac{n^2-6}2}{}y(u)\\ 
&\hskip 90pt 
+C_3\,{}x(u)^{\frac{n^2-4}2}
+C_4\,{}x(u)^{\frac{n^2-8}2}{}y(u)+\cdots+C_{n^2-1}
\end{aligned}
\end{equation}
is the \(n\)-plication polynomial.
The first several terms of this is given by
\begin{equation}\label{4.06}
\begin{aligned}
C_1&=-{\tfrac1{2}}\,n{\mu_1},\\
C_2&=-{\tfrac1{24}}\,n(n^2-2^2){\mu_1}^2
-{\tfrac1{6}}\,n(n^2-2^2){\mu_2},\\
C_3&=-{\tfrac1{48}}\,n(n^2-2^2){\mu_1}^3
-{\tfrac1{12}}\,n(n^2-2^2){\mu_2}{\mu_1}
-{\tfrac1{2}}\,n{\mu_3},\\
C_4&=-{\tfrac1{1920}}\,n(n^2-2^2)(n^2-4^2){\mu_1}^4
 -{\tfrac1{240}}\,n(n^2-2^2)(n^2-4^2){\mu_2}{\mu_1}^2\\
&\ \ -{\tfrac1{120}}\,n(n^2-2^2)(n^2+9){\mu_3}{\mu_1}
     -{\tfrac1{120}}\,n(n^2-2^2)(n^2-4^2){\mu_2}^2\\
&\ \ -{\tfrac1{60}}\,n(n^2-2^2)(n^2+9){\mu_4},\\
C_5&=-{\tfrac1{3840}}\,n(n^2-2^2)(n^2-4^2){\mu_1}^5
     -{\tfrac1{480}}\,n(n^2-2^2)(n^2-4^2){\mu_2}{\mu_1}^3\\
&\ \ -{\tfrac1{240}}\,n(n^2-2^2)(n^2+14){\mu_3}{\mu_1}^2
     -{\tfrac1{240}}\,n(n^2-2^2)(n^2-4^2){\mu_2}^2{\mu_1}\\
&\ \ -{\tfrac1{120}}\,n(n^2-2^2)(n^2+9){\mu_4}{\mu_1}
     -{\tfrac1{12}}\,n(n^2-2^2){\mu_3}{\mu_2}.
\end{aligned}
\end{equation}
These are elements of \(\mathbb{Z}[\vec{\mu}]\) if  \(n\) is even.

\newpage
\bibliography{n-plication.bib}{}
\bibliographystyle{plain}

\end{document}